\documentclass[11pt,letterpaper]{amsart}

\usepackage[letterpaper,margin=1in]{geometry}
\usepackage[all]{xy} 
\usepackage{amsmath, amssymb, amsfonts, latexsym, mdwlist, amsthm, amscd,mathabx,braket}
\usepackage{subfig}
\usepackage{graphicx}
\usepackage{wrapfig}
\usepackage{etex}
\usepackage{tikz-cd}
\usepackage[colorlinks=true, citecolor=blue, urlcolor=blue, linkcolor=blue, pagebackref]{hyperref}

\usepackage{tikz}
\usetikzlibrary{calc,trees,positioning,arrows,chains,shapes.geometric,%
    decorations.pathreplacing,decorations.pathmorphing,shapes,%
    matrix,shapes.symbols}

\tikzset{
>=stealth',
  punktchain/.style={
    rectangle,
    rounded corners,
    draw=black, thick,
    minimum height=3em,
    text centered,
    on chain},
  line/.style={draw, thick, <-},
  eLement/.style={
    tape,
    top color=white,
    bottom color=blue!50!black!60!,
    minimum width=8em,
    draw=blue!40!black!90, very thick,
    text width=10em,
    minimum height=3.5em,
    text centered,
    on chain},
  every join/.style={->, thick,shorten >=1pt},
  decoration={brace},
  tuborg/.style={decorate},
  tubnode/.style={midway, right=2pt},
}

\usepackage{paralist}
\usepackage{enumitem} 
\setdefaultenum{(1)}{(a)}{(i)}{}
\setlist[enumerate,1]{label={\upshape(\arabic*)}}
\setlist[enumerate,2]{label={\upshape(\alph*)},ref=\theenumi\upshape(\alph*)}
\setlist[enumerate,3]{label={\upshape(\roman*)},ref=\theenumi\theenumii\upshape(\roman*)}
\usepackage[capitalize]{cleveref}
\crefname{Prop}{Proposition}{Propositions}
\crefname{Thm}{Theorem}{Theorems}
\crefname{Lem}{Lemma}{Lemmas}
\crefname{enumi}{Case}{Cases}

\newcommand{\prs}[1]{\left( #1 \right)}

\newcommand{\mcal}[1]{\mathcal{#1}}

\newcommand{\mbf}[1]{\mathbf{#1}}

\DeclareMathOperator{\tf}{t.f.}
\newcommand{\id}{\mathrm{id}}
\newcommand{\eps}{\varepsilon}
\DeclareMathOperator{\rank}{rank}

\DeclareMathOperator{\Todd}{Td}

\def\ch{\mathop{\mathrm{ch}}\nolimits}
\def\Char{\mathop{\mathrm{char}}}

\def\Coh{\mathop{\mathrm{Coh}}\nolimits}
\def\codim{\mathop{\mathrm{codim}}\nolimits}

\def\deg{\mathop{\mathrm{deg}}}
\def\depth{\mathop{\mathrm{depth}}}

\def\dim{\mathop{\mathrm{dim}}\nolimits}

\def\Ext{\mathop{\mathrm{Ext}}\nolimits}

\def\lExt{\mathop{\mathcal Ext}\nolimits} 

\def\lHom{\mathop{\mathcal Hom}\nolimits}
\def\RlHom{\mathop{\mathbf{R}\mathcal Hom}\nolimits}

\def\id{\mathop{\mathrm{id}}\nolimits}

\def\sing{\mathop{\mathrm{sing}}}

\def\supp{\mathop{\mathrm{supp}}}

\def\Td{\mathop{\mathrm{Td}}\nolimits}

\def\tor{\mathop{\mathrm{tor}}\nolimits}

\def\MG13{\ensuremath{{\mathcal M}_{\Gamma_1(3)}}}
\def\tildeMG13{\ensuremath{\widetilde{\mathcal M}_{\Gamma_1(3)}}}

\def\into{\ensuremath{\hookrightarrow}}

\def\Db{\mathrm{D}^{\mathrm{b}}}

\newcommand\TFILTB[3]{%
\xymatrix@=1pc{
{0 = {#1}_0} \ar[rr]&&
{{#1}_1} \ar[rr]\ar[ld] &&
{{#1}_2} \ar[r]\ar[ld] &
{\cdots} \ar[r] & { {#1}_{#3-1}} \ar[rr] &&
{{#1}_{#3} = {#1}} \ar[ld]
\\
& *{{#2}_1} \ar@{.>}[ul] &&
{{#2}_2} \ar@{.>}[ul] & &&&
{{#2}_{{#3}}} \ar@{.>}[ul]
}}

\def\abs#1{\left\lvert#1\right\rvert}

\makeatletter
\newtheorem*{rep@theorem}{\rep@title}
\newcommand{\newreptheorem}[2]{%
\newenvironment{rep#1}[1]{%
 \def\rep@title{#2 \ref{##1}}%
 \begin{rep@theorem}}%
 {\end{rep@theorem}}}
\makeatother

\newtheorem{Thm}{Theorem}[section]
\newreptheorem{Thm}{Theorem}
\newtheorem{Prop}[Thm]{Proposition}

\newtheorem{Lem}[Thm]{Lem}

\newtheorem{Cor}[Thm]{Corollary}
\newreptheorem{Cor}{Corollary}

\newreptheorem{Con}{Conjecture}

\newtheorem*{theorem*}{Theorem}
\newtheorem*{lemma*}{Lem}
\newtheorem*{proposition*}{Proposition}
\newtheorem*{conjecture*}{Conjecture}
\newtheorem*{corollary*}{Cor}
\newtheorem*{problem*}{Problem}

\newtheorem{Thm-int}{Theorem}

\theoremstyle{definition}
\newtheorem{Def-s}[Thm]{Definition}
\newtheorem{Def}[Thm]{Definition}
\newtheorem{Rem}[Thm]{Remark}

\def\EE{\ensuremath{\mathcal E}}
\def\FF{\ensuremath{\mathcal F}}

\def\HH{\ensuremath{\mathcal H}}

\def\LL{\ensuremath{\mathcal L}}

\def\OO{\ensuremath{\mathcal O}}

\def\TT{\ensuremath{\mathcal T}}

\def\M{\ensuremath{\widetilde{M}}}
\def\X{\ensuremath{\widetilde{X}}}

\newcommand{\ignore}[1]{}
\begin{document}

\title[Bogomolov-Gieseker over RDP's]{The Bogomolov-Gieseker-Koseki inequality on surfaces with canonical singularities in arbitrary characteristic}

\author{Howard Nuer}
\author{Alan Sorani}


\begin{abstract}
This note generalizes the celebrated Bogomolov-Gieseker inequality for smooth projective surfaces over an algebraically closed field of characteristic zero to projective surfaces in arbitrary characteristic with canonical singularities.  We also generalize to this context some classical applications of the Bogomolov-Gieseker inequality.
\end{abstract}

\maketitle
\setcounter{tocdepth}{1}

\section{Introduction}
\subsection{The BG inequality and its generalizations}The celebrated Bogomolov-Gieseker (BG) inequality \cite[Theorem 3.4.1]{HL10} states that a $\mu_H$-semistable sheaf $\EE$ on a smooth polarized surface $(X,H)$ in characteristic zero has non-negative discriminant, i.e. 
\begin{equation}\label{eqn:BG}
    \Delta(\EE):=\ch_1(\EE)^2-2\ch_0(\EE)\ch_2(\EE)\geq 0.
\end{equation}
While semistable sheaves on smooth projective curves exist for any pair of rank and degree, the BG inequality puts strong restrictions on the Chern characters of semistable sheaves on smooth projective surfaces.  This is crucial in the classification problem for Chern characters of semistable sheaves on any surface.  
The inequality also gives effective results guaranteeing the stability of the restriction of a semistable sheaf to a curve in a linear system that is sufficiently divisible.
Furthermore, the classical BG inequality is fundamental in the application of vector bundle methods to the study of linear systems and vanishing theorems, as in Reider's Theorem and Kawamata-Viewegh vanishing respectively (see \cite[Corollary 4, Theorem 6]{Fri98}).   More recently, the BG inequality is needed to construct Bridgeland stability conditions on the bounded derived category $\Db(X)$ of coherent sheaves on $X$ \cite{AB13,Bri08}.

Generalizations of the BG inequality come in various flavors.  
In higher dimension, the natural generalization holds \cite[Theorem 7.3.1]{HL10}: for a $\mu_H$-semistable sheaf $\EE$ on an $n$-dimensional smooth polarized projective variety $(X,H)$ over an algebraically closed field of characteristic zero, $\Delta(\EE). H^{n-2}\geq 0$.  In positive characteristic, the BG inequality for surfaces is false \cite{Mukai:CounterExamples,Raynaud:Counterexample}, but a strong modified BG inequality holds.  Indeed, for a $\mu_H$-semistable sheaf $\EE$ on a smooth polarized surface $(X,H)$ over an algebraically closed field of positive characteristic, Koseki has shown \cite{koseki:bg-pos} that 
\begin{equation}\label{eqn:BGK}
    \Delta(\EE)+C_X\ch_0(\EE)^2\geq 0,
\end{equation}
where $C_X$ is a non-negative constant depending only on the birational equivalence class of $X$.
An improvement on previous modifications of Langer \cite{Langer:SemistablePositiveChar,Langer:HiggsPositiveChar}, the Bogomolov-Gieseker-Koseki (BGK) inequality \eqref{eqn:BGK} is strong enough to give effective restriction theorems and to construct Bridgeland stability conditions for smooth projective surfaces in positive characteristic.
A third flavor holds for an $n$-dimensional compact K\"{a}hler manifold $(X,\omega)$, 
\begin{equation}\label{eqn:BGKahler}
    \int_X\Delta(\EE)\wedge\omega^{n-2}\geq 0,
\end{equation}
as proven in \cite{BandoSiu:KahlerBG}.
A final flavor, unrelated to the direct objectives of this note, formulates for tilt-stable complexes on smooth projective threefolds a similar inequality to the BG inequality that also takes into account the third Chern character.  This is the conjectural key step in constructing Bridgeland stability conditions in dimension three \cite{BMT14a}.

In this note we explore another flavor of generalization of the BG and BGK inequalities, applying instead to singular varieties.  
Recently, Wu \cite{Wu:BGSingular} has generalized \eqref{eqn:BG} and \eqref{eqn:BGKahler} to compact, irreducible K\"{a}hler complex analytic spaces $(X,\omega)$ that are smooth in codimension two.  Note that even in characteristic zero, this result says nothing about singular projective surfaces, which is where the present note takes off.

We work in this paper over an algebraically closed field $k$ of arbitrary characteristic and consider projective surfaces with certain mild singularities, proving the following theorem:
\begin{Thm}[Main Theorem]\label{Thm:MainThm}
Let $X$ be a projective surface over an algebraically closed field $k$ with at worst rational double point singularities and let $H$ be an ample divisor on $X$.  For every $\mu_H$-semistable sheaf $\mcal{F}$ on $X$, 
$$\Delta(\FF) + C_X \ch_0(\FF)^2\geq 0.$$
\end{Thm}

Even in the absence of singularities, generalizing the BG (or BGK) inequality to compact K\"{a}hler manifolds presents difficulties since they do not necessarily satisfy the resolution property, that is, not every coherent sheaf admits a resolution by holomorphic vector bundles \cite{Voisin:Counterexamples}.  
While Chern classes are defined for holomorphic vector bundles, without the resolution property it becomes difficult to establish a well-defined theory of Chern classes for arbitrary coherent sheaves.
For singular compact K\"{a}hler complex spaces, an additional difficulty arises in considering the correct theory of cohomology to work in, as singular cohomology may no longer be the same as De Rham cohomology.

\subsection{Idea of the proof}
We side-step these difficulties by working with projective surfaces and Chow groups.  In particular, we have a closed embedding $i\colon X\into P$ of our singular surface $X$ into some projective space $P$.  
Because $P$ is smooth, it satisfies the resolution property, so we can define the Chern classes of $\FF$ via pull-back along $i$ of the Chern classes of the coherent sheaf $i_*\prs{\FF}$ on $P$.
Working with Chow also allows us work in arbitrary characteristic.  

Having settled on an appropriate cohomological venue, we encounter a second difficulty (even in characteristic zero).
The usual proof of the BG inequality for smooth surfaces makes a reduction from arbitrary $\mu_H$-semistable torsion-free sheaves to locally free ones by taking the reflexive dual.  
For singular surfaces, reflexive sheaves are no longer necessarily locally free and their symmetric powers are not well behaved, derailing attempts to push the classical proof of the BG inequality further in this setting.
Nevertheless, in the singular case we can still reduce to the case of reflexive sheaves.

To overcome this hurdle, we pull $\FF$ back along the minimal resolution of singularities $\pi\colon\X\to X$, which exists for surfaces in arbitrary characteristic.  Restricting to rational singularities guarantees that $\mbf{R}\pi_*\prs{\pi^*\FF}=\FF$, so the Chern characters of $\FF$ and $\pi^*\FF$ should be very closely related via Grothendieck-Riemann-Roch.  If $\pi^*\FF$ were $\mu_{H'}$-semistable for some ample $H'$, then we could indeed obtain the desired inequality immediately.  Unfortunately, $\pi^*\FF$ can have a non-trivial torsion subsheaf supported on the exceptional locus of $\pi$.  Taking the torsion-free quotient $\prs{\pi^*\FF}_{\tf}$, we obtain a $\mu_{\pi^*H}$-semistable locally free sheaf on $\X$ and because $\FF$ is reflexive we still have $\mbf{R}\pi_*\prs{\prs{\pi^*\FF}_{\tf}}=\FF$.
Of course, $\pi^*H$ is not ample, only nef and big, so while we cannot immediately deduce the desired inequality for $\prs{\pi^*\FF}_{\tf}$, it does follows from Koseki's invariance-of-polarization result on the smooth projective surface $\X$ \cite{koseki:bg-pos}.  By restricting to rational double point (RDP) singularities, we can use Grothendieck-Riemann-Roch to deduce the BGK inequality for $\FF$.

\subsection{Further directions}
Rational double point singularities are the natural setting arising in the study of moduli spaces of surfaces.  Indeed, in characteristic zero RDP singularities are precisely the canonical surface singularities, i.e. the singularities one encounters on the canonical models of smooth surfaces when running the Minimal Model Program (MMP) \cite[Claim 4-6-10]{Mat02} .  Our hope is that our generalization of the BGK inequality to canonical surface singularities extends the range and power of vector bundle techniques in the MMP and the study of moduli of surfaces.

The natural generalization of our result to higher dimension, at least in characteristic zero, would be for an $n$-dimensional projective variety $X$ with Gorenstein, rational singularities, as these are the two features of RDP singularities we use. 
In the setting of characteristic zero, these conditions are equivalent to having canonical singularities \cite[Corollary 5.24]{KM98}.
By taking the intersection of $n-2$ general hyperplane sections $H_i\in |H|$, we obtain a surface $S:=H_1\cap\dots\cap H_{n-2}$ with at worst RDP singularities.   So if one can prove a Mehta-Ramanathan-type restriciton theorem for varieties with at worst Gorenstein, rational singularities, then for a $\mu_H$-semistable sheaf $\FF$ on $X$, $\FF|_S$ must be $\mu_{H|_S}$-semistable and we can deduce that
$$\Delta(\FF). H^{n-2}=\Delta(\FF|_S)\geq 0,$$
as required.  A similar approach to Koseki's in positive characteristic would likely adjust this argument to work for proving the BGK inequality on such singular varieties in positive characteristic. 

In another direction, we expect that our generalization of the BGK inequality will allow for the construction of Bridgeland stability conditions on projective surfaces with RDP singularities in arbitrary characteristic.

\subsection*{Acknowledgments}
The seed for the ideas here was planted during conversations at Northeastern University between the first author and Emanuele Macr\`{i}.  He thanks Emanuele for these  productive conversations and Northeastern for providing a conducive research environment.  The authors would also like to thank Izzet Coskun and Evgeny Shinder for illuminating discussions during the writing of this paper.  

\section{Reflexive sheaves on singular surfaces}
We start by recalling the definition of rational singularities (see \cite[Definition 2.76]{Kol13}).  
\begin{Def}\label{Def:RationalSingularities}
A variety $X$ over an algebraically closed field $k$ has \emph{rational singularities} if it admits a rational resolution of singularities, i.e. a resolution $\pi:\X\to X$ such that $\mbf{R}\pi_*\prs{\OO_{\X}}\cong\OO_X$ and $\mbf{R}\pi_*\prs{\omega_{\X}}\cong\omega_X$ in the derived category $\Db(X)$.
It turns out that if $X$ admits one resolution as above, then every resolution satisfies these properties.
\end{Def}
\begin{Rem}
In characteristic zero, the second requirement that $\mbf{R}\pi_*\prs{\omega_{\X}}\cong\omega_X$ follows from the Grauert-Reimenschneider theorem, while it is necessary to state in positive characteristic.  Similarly, the fact that every resolution of a variety with rational singularities is rational is quite easy in characteristic zero, but quite hard, albeit true, in positive characteristic (see the discussion following \cite[Definition 2.76]{Kol13}).
\end{Rem}
We also recall the following easy observation:
\begin{Lem}\label{lemma:pushforward-tf}
Let $\pi \colon Y \to X$ be a dominant map between integral schemes, and let $M$ be a torsion-free sheaf on $Y$. Then $\pi_*\prs{M}$ is torsion-free.
\end{Lem}

The following well-known result (see \cite[Lemma and Definition 2.2]{Esnault:Reflexive}), whose proof we include for the sake of completeness, is one fundamental key to our approach.  
\begin{Prop}\label{Prop:SheafCorrespondence}
Let $X$ be a projective surface with at worst rational singularities and with a resolution of singularities $\pi\colon\X\to X$.
Let $\M$ be a sheaf on $\X$.  There exists a reflexive sheaf $M$ on $X$ such that $\M=\prs{\pi^*M}_{\tf}:=\pi^*M/(\pi^*M)_{\tor}$ if and only if the following conditions are satisfied:
\begin{enumerate}
\item $\M$ is locally free.
\item $\pi^*\pi_*\M\to\M$ is surjective.
\item $\mbf{R}^1\pi_*(\M^{\vee}\otimes\omega_{\X})=0$.
\end{enumerate}
In this case $\mbf{R}\pi_*\M=M$ and $\pi_*(\M^{\vee})=M^{\vee}$.
\end{Prop}
\begin{proof}
Let $M$ be a reflexive sheaf on $X$ and let us verify that $\M=\prs{\pi^*M}_{\tf}$ satisfies the conditions (1)-(3) of the proposition.
Because $X$ has rational singularities, the projection formula gives $$\mbf{R}\pi_*L\pi^*M\cong M\otimes \mbf{R}\pi_*(\OO_{\X})\cong M\otimes\OO_X=M.$$ 
We use the fourth quadrant spectral sequence $$E_2^{p,q}=\mbf{R}^p\pi_*L^q\pi^*M\Rightarrow \HH^{p+q}(\mbf{R}\pi_*L\pi^*M)=\HH^{p+q}(M)=
\begin{cases}
M,& p+q=0\\
0,& \text{otherwise}
\end{cases}.$$
Since $E_2^{p,q}=0$ for $p>1$ as the fibers of $\pi$ are at most one dimensional, the spectral sequence degenerates at the $E_2$-page.
Thus it follows immediately that $\mbf{R}^1\pi_*\pi^*M=0$, so there exists a short exact sequence $$0\to \mbf{R}^1\pi_*L^{-1}\pi^*M\to M\to\pi_*\pi^*M\to 0.$$
Since $\mbf{R}^1\pi_*F$ is supported in $X_{\sing}$ for any $F\in\Coh(\X)$, it must be a torsion sheaf, and as $M$ is reflexive, it is in particular torsion-free, so $\mbf{R}^1\pi_*L^{-1}\pi^*M=0$.  Thus $\pi_*\pi^*M=M$.

Now let $T:=(\pi^*M)_{\tor}$ be the torsion part of $\pi^*M$, and consider the torsion-free sheaf $\M$ fitting into the short exact sequence \begin{equation}\label{eqn: torsion exact sequence}0\to T\to \pi^*M\to \M\to 0.\end{equation}
Pushing forward this exact sequence, we get \begin{equation}\label{eqn: sequence on low order terms}0\to \pi_*T\to\pi_*\pi^*M\to\pi_*\M\to \mbf{R}^1\pi_*T\to \mbf{R}^1\pi_*\pi^*M\to \mbf{R}^1\pi_*\M\to 0.\end{equation}
From above we have $\pi_*\pi^*M=M$ and $\mbf{R}^1\pi_*\pi^*M=0$, so $\mbf{R}^1\pi_*\M=0$ and $\pi_*T$ is a subsheaf of $M$.
But as $T$ must be supported above the singular locus $X_{\sing}$, $\pi_*T$ is torsion and thus $\pi_*T=0$ from the torsion-freeness of $M$. So the sequence \eqref{eqn: sequence on low order terms} becomes $$0\to M\to\pi_*\M\to \mbf{R}^1\pi_*T\to 0.$$  Because $\M$ is torsion-free, $\pi_*\M$ is as well, so we get \begin{equation}\label{eqn: surface reflexive dual sequence}M\into\pi_*\M\into (\pi_*\M)^{\vee\vee},\end{equation} which becomes an isomorphism when restricted to $X-X_{\sing}$.  As both $M$ and $(\pi_*\M)^{\vee\vee}$ are reflexive and $\codim X_{\sing}\geq 2$, the composition \eqref{eqn: surface reflexive dual sequence} is in fact an isomorphism.  Thus $\mbf{R}^1\pi_*T=0$ and $\pi_*\M=M$, which implies that $\mbf{R}\pi_*\M=M$ in $\Db(X)$.
This immediately gives condition (2) of the proposition since $\M$ is by definition a quotient of $\pi^*\pi_*\M=\pi^*M$.

To see conditions (1) and (3) of the proposition, we use Grothendieck duality and the fact that $\mbf{R}\pi_*\M=M$ to get $$\mbf{R}\pi_*(\RlHom(\M,\omega_{\X}))\cong\RlHom(\mbf{R}\pi_*\M,\omega_X)\cong\RlHom(M,\omega_X).$$  Now $M$ being reflexive is equivalent to $\HH_Y^1(M)=0$ for every closed $Y\subset X$ with $\codim Y\geq 2$, where $\HH_Y^1$ denotes the local cohomology sheaf \cite[Proposition 1.6(iv)]{Har80}.  Moreover, by the relation between depth and the vanishing of the local cohomology sheaf \cite[Theorem 3.8]{Har67}, we see that this vanishing is equivalent to $\depth_Y(M)\geq 2$ for all $Y$ as above.  As rational singularities are CM \cite[Proposition 2.77]{Kol13}, we may use \cite[Proposition 21.12]{Eis95} to get that for every closed point $x\in X$, $\lExt^i(M,\omega_X)_x\cong\Ext^i_{\OO_x}(M_x,\omega_{X,x})=0$ for all $i>\dim(\OO_x)-\depth_{\mathfrak m_x}(M_x)=0$.  Thus $\lExt^1(M,\omega_X)=0$ so that $\HH^1(\mbf{R}\pi_*(\RlHom(\M,\omega_{\X}))=0$ as well.  Now using the spectral sequence $$E_2^{p,q}=\mbf{R}^p\pi_*(\lExt^q(\M,\omega_{\X}))\Rightarrow\HH^{p+q}(\mbf{R}\pi_*(\RlHom(\M,\omega_{\X}))),$$ we see that this implies that $\mbf{R}^1\pi_*(\M^{\vee}\otimes\omega_{\X})=0$ (condition (3) of the proposition) and that $\pi_*(\lExt^1(\M,\omega_{\X}))=0$.  But as $\X$ is smooth and $\M$ is torsion-free, $\codim\lExt^1(\M,\omega_{\X})\geq 2$.  The vanishing of $\pi_*(\lExt^1(\M,\omega_{\X}))$ implies that $\lExt^1(\M,\omega_{\X})=0$ as well.  Thus $\M$ is reflexive, and as $\X$ is a smooth surface, it must be locally free as well, completing the proof of the ``only if'' direction of the proposition.

In the converse direction, suppose that $\M$ satisfies conditions (1)--(3) of the proposition.  Since $\M$ is locally free, we use Grothendieck duality to get $$\mbf{R}\pi_*(\M^{\vee}\otimes\omega_{\X})\cong \mbf{R}\pi_*(\RlHom(\M,\omega_{\X}))\cong \RlHom(\mbf{R}\pi_*\M,\omega_X),$$ and using the spectral sequence $$E_2^{p,q}=\lExt^p(\HH^{-q}(\mbf{R}\pi_*\M),\omega_X)=\lExt^p(\mbf{R}^{-q}\pi_*\M,\omega_X)\Rightarrow\lExt^{p+q}(\mbf{R}\pi_*\M,\omega_X),$$ the vanishing of $\mbf{R}^1\pi_*(\M^{\vee}\otimes\omega_{\X})$ gives $\ker(d_2^{1,0})=0$ so $$\lExt^1(\pi_*\M,\omega_X)\into\lExt^3(\mbf{R}^1\pi_*\M,\omega_X).$$  But by \cite[Proposition 21.12(a)]{Eis95}, $\lExt^3(\mbf{R}^1\pi_*\M,\omega_X)=0$ since the injective dimension of $\omega_X$ is 2.  So $\lExt^1(\pi_*\M,\omega_X)=0$, and thus by \cite[Corollary 3.5.11(b)]{BH93}, $1>2-\depth_x(\pi_*\M)$ at any closed point $x\in X$, i.e. $\depth_x(\pi_*\M)\geq 2$, so $\pi_*\M$ is reflexive.  

Since $\pi^*\pi_*\M\to\M$ is surjective, and an isomorphism away from the exceptional locus, the kernel must be torsion, so letting $M=\pi_*\M$, we see that $\M=\pi^*M/(\pi^*M)_{\tor}$, as required.

Finally, we prove that $\pi_*(\M^{\vee})=M^{\vee}$.  First note that by \cref{Def:RationalSingularities}, it follows that $$\mbf{R}\pi_*(L\pi^*M\otimes\omega_{\X})\cong M\otimes \mbf{R}\pi_*(\omega_{\X})\cong M\otimes \omega_X.$$  Furthermore, using Grothendieck duality we get \begin{align}
\begin{split}
M^{\vee}&=\lHom(M,\OO_X)=\HH^0\RlHom(M,\OO_X)=\HH^0\RlHom(M\otimes\omega_X,\omega_X)\\
&=\HH^0\RlHom(\mbf{R}\pi_*(L\pi^*M\otimes\omega_{\X}),\omega_X)\cong\HH^0\mbf{R}\pi_*\RlHom(L\pi^*M\otimes\omega_{\X},\omega_{\X})\\
&=\HH^0\mbf{R}\pi_*\RlHom(L\pi^*M,\OO_{\X}).
\end{split}
\end{align}

The spectral sequence $$E_2^{p,q}=\mbf{R}^p\pi_*\HH^q\RlHom(L\pi^*M,\OO_{\X})\Rightarrow\HH^{p+q}\mbf{R}\pi_*\RlHom(L\pi^*M,\OO_{\X})$$ gives the short exact sequence $$0\to \mbf{R}^1\pi_*\HH^{-1}\RlHom(L\pi^*M,\OO_{\X})\to M^{\vee}\to\pi_*\HH^0\RlHom(L\pi^*M,\OO_{\X})\to 0,$$
while the spectral sequence $$E_2^{p,q}=\lExt^p(L^{-q}\pi^*M,\OO_{\X})\Rightarrow\HH^{p+q}\RlHom(L\pi^*M,\OO_{\X})$$ shows that $\HH^{-1}\RlHom(L\pi^*M,\OO_{\X})=0$ and $\HH^0\RlHom(L\pi^*M,\OO_{\X})=(\pi^*M)^{\vee}$.  Thus $M^{\vee}\cong\pi_*((\pi^*M)^{\vee})$, and applying $\lHom(-,\OO_{\X})$ to the short exact sequence \eqref{eqn: torsion exact sequence} gives $(\pi^*M)^{\vee}\cong\M^{\vee}$, so $$\pi_*(\M^{\vee})\cong M^{\vee},$$ as required.
\end{proof}

\section{The Bogomolov-Gieseker-Koseki Inequality on singular surfaces}

\subsection{The Bogomolov-Gieseker-Koseki Inequality on smooth surfaces}
Let us recall Koseki's modified version of the Bogomolov-Gieseker inequality for smooth surfaces in arbitrary characteristic.
\begin{Thm}[{\cite[Theorem 3.5]{koseki:bg-pos}}]\label{Thm:BG-smooth}
Let $S$ be a smooth projective surface defined over an algebraically closed field $k$. Then there exists a constant $C_S\geq 0$, depending only on the characteristic of $k$ and the birational equivalence class of $S$, satisfying the following condition: for every numerically non-trivial nef divisor $H$ on $S$ and $\mu_H$-semistable torsion free sheaf $\EE \in\Coh(S)$ on $S$, the inequality
$$\Delta(\EE) + C_S \ch_0(\EE)^2\geq 0$$ holds. Explicitly, we can take the constant $C_S$ as follows:
\begin{enumerate}
    \item If $\Char k>0$ and $S$ is a minimal surface of general type, then $C_S=2+5K_S^2-\chi\prs{\OO_S}$.
    \item If $\Char k>0$, $\kappa\prs{S}=1$, and $S$ is quasi-elliptic, then $C_S=2-\chi(\OO_S)$
    \item Otherwise, set $C_S=0$.
\end{enumerate}
\end{Thm}

In the course of proving the above theorem, Koseki establishes some results that will be of use to us.

\begin{Lem}[{\cite[Lemma 3.7]{koseki:bg-pos}}] \label{lemma:bg-invariance}
Let $S$ be a smooth projective surface, $L$ a numerically non-trivial nef divisor. Suppose that we have an exact sequence
\[0 \to \mcal{E}_1 \to \mcal{E} \to \mcal{E}_2 \to 0\]
of torsion-free sheaves on $S$ with $\mu_L\prs{\mcal{E}_1} = \mu_L\prs{\mcal{E}} = \mu_L\prs{\mcal{E}_2}$. Then we have
\[\frac{\Delta\prs{\mcal{E}}}{\ch_0\prs{\mcal{E}}} \geq \frac{\Delta\prs{\mcal{E}_1}}{\ch_0\prs{\mcal{E}_1}} + \frac{\Delta\prs{\mcal{E}_2}}{\ch_0\prs{\mcal{E}_2}} \text{.}\]
\end{Lem}

For a positive integer $r$, we say $T^1\prs{r}$ holds for a pair $(S,L)$ of a smooth projective surface and a numerically non-trivial nef divisor $L$ if \cref{Thm:BG-smooth} holds for $(S,L)$ whenever $\ch_0\prs{\mcal{E}} \leq r$.

\begin{Prop}[{\cite[Proposition 3.8]{koseki:bg-pos}}]\label{proposition:invariance}
Let $S$ be a smooth projective surface and $r \geq 2$ an integer. Suppose that $T^1\prs{r}$ holds for $(S,L)$ for some numerically non-trivial nef divisor $L$. Then it also holds for $(S,L')$ for every choice of numerically non-trivial nef divisor $L'$.
\end{Prop}

\subsection{Relating invariants upstairs and downstairs}

Now suppose that $X$ is a projective surface with at worst rational singularities and let $\pi\colon\X\to X$ be its minimal resolution.  

\begin{Prop}\label{proposition:pushforward}
Let $\mcal{E} \in \Coh\prs{\X}$ be torsion-free and $H$ a numerically non-trivial nef divisor on $X$.
The following statements hold.

\begin{enumerate}
    \item The sheaf $\pi_*\prs{\mcal{E}} \in \Coh\prs{X}$ is torsion-free, and the sheaf $\mathbf{R}^1\pi_*\prs{\mcal{E}}$ is supported on a zero-dimensional subscheme.

    \item We have $\ch_0\prs{\pi_*\mcal{E}} = \ch_0\prs{\mcal{E}}$. \label{enumerate:psuhforward-ch0}

    \item We have $H .\ch_1\prs{\pi_*\mcal{E}} = \pi^* H. \ch_1\prs{\mcal{E}}$. \label{enumerate:pushforward-ch1}

    \item If $\mcal{F} \in \Coh\prs{X}$ is reflexive and $\mu_H$-semistable, then $\prs{\pi^* \mcal{F}}_{\tf}$ is $\mu_{\pi^*H}$-semistable.

\end{enumerate}
\end{Prop}

\begin{proof}
    \begin{enumerate}
        \item
        Firstly, $\pi$ is dominant so by \cref{lemma:pushforward-tf},  $\pi_*\prs{\mcal{E}}$ is torsion-free. As $\pi$ is an isomorphism outside of the singular locus of $X$, $\mbf{R}^1\pi_*\prs{\mcal{E}}$ vanishes there, so it must be supported on the singular locus of $X$, a zero-dimensional subscheme.

        \item
        Since $\ch_0$ of a sheaf is its generic rank and $\pi$ is an isomorphism over the generic point, the equality follows.
        
        \item

        Let $\eps \colon \pi^* \pi_* \to \id_{\Coh\prs{\tilde{X}}}$
        be the counit of the adjunction $\pi^* \dashv \pi_*$.

        We get an exact sequence
        \[0 \to K \to \pi^* \pi_* \mcal{E} \xrightarrow{\eps_{\mcal{E}}} \mcal{E} \to C \to 0\]
        with $K, C$ being the kernel and cokernel of $\eps_{\mcal{E}}$, respectively.
        Using the additivity of the Chern character we get that
        \[\ch\prs{K} - \ch\prs{\pi^* \pi_* \mcal{E}} + \ch\prs{\mcal{E}} - \ch\prs{C} = 0 \text{,}\]
        so that
        \[\ch\prs{\mcal{E}} = \ch\prs{\pi^* \pi_* \mcal{E}} - \ch\prs{K} + \ch\prs{C} \text{.}\]
        Hence
        \begin{align*}
            \pi^* H . \ch_1\prs{\mcal{E}} &= \pi^* H . \prs{\ch_1\prs{\pi^* \pi_* \mcal{E}} - \ch_1\prs{K} + \ch_1\prs{C}} \text{.}
        \end{align*}
        Since $\pi$ is an isomorphism outside of the exceptional divisor $E$, we get that $\supp\prs{K}, \supp\prs{C} \subseteq E$, so $\pi^* H . \ch_1\prs{K} = \pi^* H . \ch_1\prs{C} = 0$. Therefore, from the above equality and the fact that $\deg\prs{\pi}=1$ it follows that
        \begin{align*}
            \pi^* H . \ch_1\prs{\mcal{E}} &= \pi^* H . \ch_1\prs{\pi^* \pi_* \mcal{E}}
            \\&=
            \pi^* H . \pi^* \ch_1\prs{\pi_* \mcal{E}}
            \\&=
            \deg\prs{\pi} \prs{H . \ch_1\prs{\pi_* \mcal{E}}}
            \\&=
            H.\ch_1\prs{\pi_* \mcal{E}} \text{.}
        \end{align*}
        
        \item
        Let $\mcal{G} \subseteq \prs{\pi^* \mcal{F}}_{\tf}$ with $\rank\prs{\mcal{G}} < \rank\prs{\mcal{F}}$.
        By part \ref{enumerate:psuhforward-ch0} above, we have $\rank\prs{\pi^* \mcal{F}} = \rank\prs{\mcal{F}}$ and $\rank\prs{\mcal{G}} = \rank\prs{\pi_*\mcal{G}}$. Then 
        \[\rank\prs{\prs{\pi^* \mcal {F}}_{\tf}} = \rank\prs{\pi^* \mcal{F}} = \rank\prs{\mcal{F}}\] and \[\pi_* \mcal{G} \subseteq \pi_* \prs{\prs{\pi^* \mcal{F}}_{\tf}}  \text{,}\]
        where we showed that $\pi_* \prs{\prs{\pi^* \mcal{F}}_{\tf}} = \mcal{F}$ for a reflexive sheaf $\FF$ in \cref{Prop:SheafCorrespondence}.
        
        In particular, $\rank\prs{\pi_* \mcal{G}} < \rank\prs{\mcal{F}}$, so by assumption we have
        \[\mu_H\prs{\pi_*\mcal{G}} \leq \mu_H\prs{\mcal{F}}\]
        or equivalently
        \begin{equation}\label{equation:stability-below}
            \frac{\ch_1\prs{\pi_*\mcal{G}} . H}{\rank\prs{\pi_*\mcal{G}}} \leq \frac{\ch_1\prs{\mcal{F}} . H}{\rank\prs{\mcal{F}}} \text{.}
        \end{equation}
        
        Now, since $\mcal{F} \cong \pi_* \prs{\pi^* \mcal{F}}_{\tf}$, part \ref{enumerate:pushforward-ch1} above gives
        \begin{align*}
            H . \ch_1\prs{\pi_* \mcal{G}} &= \pi^* H . \ch_1\prs{\mcal{G}} \text{,} \\
            H. \ch_1\prs{\mcal{F}} &= H . \ch_1\prs{\pi_* \prs{\pi^* \mcal{F}}_{\tf}} = \pi^* H . \ch_1\prs{\prs{\pi^* \mcal{F}}_{\tf}} \text{.}
        \end{align*}
        Hence \eqref{equation:stability-below} gives
        \[\mu_{\pi^* H}\prs{\mcal{G}} = \frac{\pi^* H \ch_1\prs{\mcal{G}}}{\rank\prs{\mcal{G}}} \leq \frac{\pi^* H \ch_1\prs{\prs{\pi^* \mcal{F}}_{\tf}}}{\rank\prs{\prs{\pi^* \mcal{F}}_{\tf}}} = \mu_{\pi^* H}\prs{\prs{\pi^*\mcal{F}}_{\tf}} \text{.}\]

        Thus $\prs{\pi^*\mcal{F}}_{\tf}$ is $\mu_{\pi^*H}$-semistable.
    \end{enumerate}
\end{proof}
 
\begin{Lem}\label{lemma:ch2 of torsion part}
       If $\mcal{F} \in \Coh\prs{X}$ is reflexive, then
    \begin{equation}\label{equation:discriminant-resolution}
        \ch_2\prs{\prs{\pi^* \mcal{F}}_{\tor}} = \frac{1}{2} K_{\tilde{X}} . \ch_1\prs{\prs{\pi^* \mcal{F}}_{\tor}} \text{.}
    \end{equation}
\end{Lem}

\begin{proof}
We note first that by \cite[Example 18.3.4(b)]{Fulton}, 
\begin{equation}\label{equation:todd-resolution}
\Td(X)=\pi_*\Td(\X)+\sum_{p\in X_{\sing}} n_p=\pi_*\Td(\X), 
\end{equation}
where $n_p=l((\mbf{R}^1\pi_*\OO_{\X})_p)=0$ vanishes because $X$ has rational singularities.

Let $\mcal{E} := \prs{\pi^* \mcal{F}}_{\tf}$.
By Grothendieck-Riemann-Roch for singular varieties \cite[Theorem 18.3(1)]{Fulton}
    \[\tau_X(\mbf{R}\pi_*\mcal{E})=\pi_*\prs{\tau_{\tilde{X}}\prs{\mcal{E}}}=\pi_*\prs{\ch\prs{\mcal{E}}.\Todd\prs{\tilde{X}}} \text{,}\]
    since $\tilde{X}$ is smooth.
    Since $\mcal{F}$ is reflexive, we get from \Cref{Prop:SheafCorrespondence} that $\mbf{R}\pi_*\mcal{E}=\mcal{F}$, so 
    \[\tau_X\prs{\mbf{R}\pi_*\mcal{E}}=\tau_X\prs{\mcal{F}}\text{.}\]
    Let $i:X\hookrightarrow P$ be a closed embedding of $X$ into some (smooth) projective space $P$.  Since $X$ has at worst local complete intersection singularities, $i$ is an l.c.i. morphism so that by \cite[Theorem 18.3(4)]{Fulton} we have 

    \begin{align*}
        \tau_X\prs{\mcal{F}}&=\tau_X\prs{i^*i_*\mcal{F}}=\mathrm{td}\prs{T_i}i^*\tau_P\prs{i_*\mcal{F}}=\mathrm{td}\prs{N_{X/P}}^{-1}i^*\prs{\mathrm{td}\prs{T_P}.\ch\prs{i_*\mcal{F}}}\\
        &=\mathrm{td}\prs{T_X}.i^*\prs{\ch\prs{i_*\mcal{F}}}=\Todd\prs{X}.\ch\prs{i^*i_*\mcal{F}}=\Todd\prs{X}.\ch\prs{\mcal{F}}\text{.}
        \end{align*}
        It follows from \eqref{equation:todd-resolution} 
        that $\Todd\prs{X}=\pi_*\Todd\prs{\tilde{X}}$ and thus from the projection formula we get

        \[\tau_X\prs{\mcal{F}}=\Todd\prs{X}.\ch\prs{\mcal{F}}=\pi_*\prs{\Todd\prs{\tilde{X}}.\pi^*\ch\prs{\mcal{F}}}=\pi_*\prs{\Todd\prs{\tilde{X}}.\ch\prs{\pi^*\mcal{F}}}.\]
  Putting everything together we get
  \[\pi_*\prs{\Todd\prs{\tilde{X}}.\ch\prs{\mcal{E}}}=\pi_*\prs{\Todd\prs{\tilde{X}}.\ch\prs{\pi^*\mcal{F}}},\]
  from which we get that 
    \[\pi_*\prs{\ch\prs{\prs{\pi^* \mcal{F}}_{\tor}} . \Todd\prs{\tilde{X}}}= 0.
    \]
    Comparing terms in codimension $2$, and using the fact that $\rank\prs{\prs{\pi^* \mcal{F}}_{\tor}} = 0$ we get that
    $$ - \frac{1}{2} \ch_1\prs{\prs{\pi^* \mcal{F}}_{\tor}} . K_{\tilde{X}} + \ch_2\prs{\prs{\pi^* \mcal{F}}_{\tor}}  = 0$$
    and thus the result.
\end{proof}
\begin{Rem}
In general in this paper, we use a closed embedding $i\colon X\into P$ into a smooth projective space to define the Chern character of a coherent sheaf on the singular variety $X$ as in the above lemma.  By \cite[Proposition 18.3.1]{Fulton}, this is independent of the choice of embedding.
\end{Rem}

\begin{Prop}\label{lemma:bg-resolution}
    Let $\mcal{F} \in \Coh\prs{X}$ be reflexive and let $\mcal{E} = \prs{\pi^* \mcal{F}}_{\tf}$ with the above notation.
    Assume further that $X$ has at worst rational double point singularities. 
    Then
    \begin{equation}\label{equation:singular-BG}
        \Delta\prs{\mcal{F}} \geq \Delta\prs{\mcal{E}} \text{.}
    \end{equation}
\end{Prop}

\begin{proof}
    Let $Z := \ch_1\prs{\prs{\pi^* \mcal{F}}_{\tor}}$. 
    We have
    \begin{align*}
        \Delta\prs{\mcal{F}} &= \ch_1\prs{\mcal{F}}^2 - 2 \ch_0\prs{\mcal{F}} . \ch_2\prs{\mcal{F}}
        \\&= \prs{\ch_1\prs{\pi^* \mcal{F}}}^2 - 2 \ch_0\prs{\pi^* \mcal{F}}  \ch_2\prs{\pi^* \mcal{F}} \text{.}
    \end{align*}
    Using the additivity of $\ch$ on short exact sequences, and $\ch_0\prs{\pi^* \mcal{F}} = \ch_0\prs{\prs{\pi^* \mcal{F}}_{\tf}} = \ch_0\prs{\mcal{E}}$, we get
    \begin{align*}
        \Delta\prs{\mcal{F}} &= \prs{Z + \ch_1\prs{\mcal{E}}}^2 - 2 \ch_0\prs{\mcal{E}}  \prs{\ch_2\prs{\mcal{E}} + \ch_2\prs{\prs{\pi^* \mcal{F}}_{\tor}}}
        \\&=
        Z^2 + 2 \ch_1\prs{\mcal{E}} . Z + \ch_1\prs{\mcal{E}}^2 - 2 \ch_0\prs{\mcal{E}}  \ch_2\prs{\mcal{E}} - 2\ch_0\prs{\mcal{E}}  \ch_2\prs{\prs{\pi^* \mcal{F}}_{\tor}}
        \\&=
        Z^2 + 2 \ch_1\prs{\mcal{E}} . Z + \ch_1\prs{\mcal{E}}^2 - 2 \ch_0\prs{\mcal{E}}  \ch_2\prs{\mcal{E}} - \ch_0\prs{\mcal{E}}  K_{\tilde{X}} . \ch_1\prs{\prs{\pi^* \mcal{F}}_{\tor}} \text{,}
    \end{align*}
    where the last equality follows from \Cref{lemma:ch2 of torsion part}.
    As $\ch_1\prs{\mcal{E}} = \ch_1\prs{\pi^* \mcal{F}} - Z$, it follows that
    \begin{align*}
        \ch_1\prs{\mcal{E}} . Z = \ch_1\prs{\pi^* \mcal{F}} . Z - Z^2 = -Z^2 \text{.}
    \end{align*}
    Thus
    \begin{align*}
        \Delta\prs{\mcal{F}} = -Z^2 + \Delta\prs{\mcal{E}} - \ch_0\prs{\mcal{E}} . K_{\tilde{X}} . Z \text{.}
    \end{align*}
    Since $X$ has at most rational double point singularities, $K_{\tilde{X}}=\pi^*K_X$ by \cite[Theorem 4-6-2]{Mat02}, so 
    $$K_{\tilde{X}} . Z = \pi^*K_X.Z=K_X.\pi_*(Z)=K_X.0=0$$
    as $Z$ is exceptional.
    Therefore
    \begin{align*}
        \Delta\prs{\mcal{F}} = -Z^2 + \Delta\prs{\mcal{E}}\geq \Delta\prs{\EE} \text{,}
    \end{align*}
    since $Z$ is exceptional so that $Z^2 \leq 0$.
\end{proof}
\subsection{Putting it all together}
\begin{Thm}
    Let $X$ be a projective surface over an algebraically-closed field $k$ with at worst rational double point singularities and let $H$ be an ample divisor on $X$.  For every $\mu_H$-semistable sheaf $\mcal{F}$ on $X$, 
\begin{equation}
   \Delta(\FF) + C_X \ch_0(\FF)^2\geq 0
\end{equation}
\end{Thm}
\begin{proof}
By the standard argument, it suffices to prove the theorem assuming $\FF$ is reflexive.  Now let $\mcal{E} = \prs{\pi^* \mcal{F}}_{\tf}$, which is $\mu_{\pi^* H}$-semistable by \Cref{proposition:pushforward}. Since $\tilde{X}$ is a smooth projective surface, \cref{Thm:BG-smooth} asserts that $T^1\prs{r}$ holds for all $r \geq 1$ for any ample divisor $H'$ on $\tilde{X}$.  Since ample divisors are certainly numerically non-trivial and nef, \Cref{proposition:invariance} then implies that $T^1\prs{r}$ holds for $\pi^* H$, so that
    \[\Delta\prs{\mcal{E}} +C_{\X}\ch_0(\EE)^2 \geq 0 \text{.}\]
    By \Cref{lemma:bg-resolution} 
    \[\Delta\prs{\mcal{F}} \geq \Delta\prs{\mcal{E}}  \text{,}\]
    so from $C_X=C_{\X}$ and $\ch_0(\FF)=\ch_0(\EE)$ we get
    $$\Delta(\FF)+C_X\ch_0(\FF)^2=\Delta(\FF)+C_{\X}\ch_0(\EE)^2\geq\Delta(\EE)+C_{\X}\ch_0(\EE)^2\geq 0,$$
    as required.
\end{proof}
\section{Applications}
We generalize some standard consequences of the BG inequality.  Although the modifications to the classical proofs necessary for the singular case are minor, we provide complete proofs for the sake of completeness. 
\begin{Cor}\label{corollary:bg-cor-1}
    Suppose $X$ is a projective surface with at worst rational double point singularities and $\mcal{F}$ is a rank 2 reflexive sheaf such that $\Delta\prs{\mcal{F}}+4C_X<0$.  Then there is a unique rank 1 reflexive subsheaf $\mcal{O}_X(D)$ of $\mcal{F}$ and an exact sequence 
    \[0\to\mcal{O}_X\prs{D}\to\mcal{F}\to\mcal{M}\to 0,\]
    where $\mcal{M}$ is a rank 1 torsion-free sheaf with $c_1\prs{\mcal{M}}=D'$, and $D,D'$ are Weil divisors such that $(D-D')^2>4C_X$ and $(D-D').H>0$ for every ample divisor $H$.
\end{Cor}
\begin{proof}
    Fix an ample divisor $H_0$.  \cref{Thm:MainThm} implies that $\mcal{F}$ is not $\mu_{H_0}$-semistable.  The maximally destabilizing subsheaf of $\FF$ with respect to $\mu_{H_0}$-stability is saturated and thus reflexive because $\FF$ itself is.  As a rank one reflexive sheaf, it must be isomorphic to $\mcal{O}_X(D)$ for some Weil divisor $D$ by \cite[Proposition 2.8]{Hartshorne:GeneralizedDivisors}.  Thus we have a short exact sequence
    \begin{equation}\label{eq:destabilizing}
        0 \to \OO_X(D) \to \mcal{F} \to \mcal{M} \to 0,
    \end{equation}
    with $\mcal{M}$ a rank 1 torsion-free sheaf.

Now observe that $\mcal{M}$ injects into its double-dual which is then another rank 1 reflexive sheaf, necessarily of the form $\mcal{O}_X(D')$, where $D'=c_1(\mcal{F})-D$.  From the instability of \eqref{eq:destabilizing}, we have $D.H_0>D'.H_0$, so $\prs{D-D'}.H_0>0$.  Moreover, we have a short exact sequence 
\[0\to\mcal{M}\to\mcal{O}_X\prs{D'}\to\mcal{T}\to 0,\]
where $\mcal{T}$ is a zero-dimensional sheaf so that $\ch_2\prs{\mcal{T}}\geq 0$.  It follows that 
\begin{align}\label{eq:2-step HNfiltration}
\begin{split}
    0&>\Delta\prs{\mcal{F}}+4C_X=c_1\prs{\mcal{F}}^2-4\ch_2\prs{\mcal{F}}+4C_X\\
    &=\prs{D+D'}^2-4\prs{\ch_2\prs{\mcal{O}\prs{D}}+\ch_2\prs{\mcal{M}}}+4C_X\\
    &=\prs{D^2+2D.D'+D'^2}-4\prs{\frac{1}{2}D^2+\frac{1}{2}D'^2-\ch_2\prs{\mcal{T}}}+4C_X\\
    &=-\prs{D-D'}^2+4\ch_2\prs{\mcal{T}}+4C_X\geq4C_X-\prs{D-D'}^2,
    \end{split}
\end{align}
so $\prs{D-D'}^2>4C_X$ as claimed.

Thus $(D-D')^2>0$,  and since $(D-D').H_0>0$ for the one ample divisor $H_0$, it follows from \cite[Ch. 1, Lemma 19]{Fri98} that $(D-D').H>0$ for every ample divisor $H$.
But then \eqref{eq:destabilizing} is a destabilizing short exact sequence for any ample divisor $H$, so the uniqueness of the Harder-Narasimhan filtration gives the uniqueness of the rank 1 reflexive subsheaf $\mcal{O}_X\prs{D}\subset\mcal{F}$, as required. 
\end{proof}

\begin{Cor}
    Let $X$ be a projective surface with at worst rational double point singularities, let $H$ be an ample Cartier divisor on $X$ and let $\mcal{F}$ be a $\mu_H$-stable rank $2$ reflexive sheaf on $X$ with $\Delta\prs{\mcal{F}} = p$. Then for all integers $k \geq\max\prs{\sqrt{\frac{p+4C_X}{H^2}}, p}$ and every integral curve $C \in \abs{k H}$, the sheaf $\left.\mcal{F}\right|_{C}$ is stable.
\end{Cor}

\begin{proof}
    Let us observe first that $\FF|_C$ is torsion-free.  Indeed, in the short exact sequence
    $$0\to \FF\prs{-C}\to\FF\to\FF|_C\to 0$$
    both $\FF$ and $\FF\prs{-C}$ are reflexive (the latter since $H$ is Cartier), so $\FF|_C$ must be a pure one-dimensional sheaf by \cite[Corollary 1.5]{Har80}.

    Now suppose that $\left. \mcal{F} \right|_C$ is not stable. 
    Then there is a rank one torsion-free quotient sheaf $\mcal{L}$ of $\left. \mcal{F} \right|_C$ with $\mu\prs{\mcal{L}} < \mu\prs{\left. \mcal{F} \right|_C}$.
    We get that
    \[\frac{\deg\prs{\mcal{L}}}{1} < \frac{\deg\prs{\left. \mcal{F} \right|_C}}{2} \text{,}\]
    so that
    \[2 \deg\prs{\mcal{L}} < \deg\prs{\left. \mcal{F} \right|_C} = \ch_1\prs{\mcal{F}} . kH \text{.}\]
    Let $\iota \colon C \hookrightarrow X$ be the inclusion and let $\mcal{F}'$ be the kernel of the surjection $\mcal{F} \to \iota_* \mcal{L}$, so that there is an exact sequence
    \[0 \to \mcal{F}' \to \mcal{F} \to \iota_* \mcal{L} \to 0 \text{.}\]
    By \cite[Chapter 2, Lemma 16]{Fri98},
    \[\ch_1\prs{\mcal{F}'} = \ch_1\prs{\mcal{F}} - kH\]
    and
    \[\ch_2\prs{\mcal{F}'} = \ch_2\prs{\mcal{F}} + \frac{1}{2} \prs{kH}^2 - \deg\prs{\mcal{L}} \text{.}\]

    We get
    \begin{align*}
        \Delta\prs{\mcal{F}'} &= \ch_1\prs{\mcal{F}'}^2 - 4 \ch_2\prs{\mcal{F}'}
        \\&=
        \ch_1\prs{\mcal{F}}^2 - 2 \ch_1\prs{\mcal{F}}.kH + \prs{kH}^2 - 4 \ch_2\prs{\mcal{F}} -2 \prs{kH}^2 + 4 \deg\prs{\mcal{L}}
        \\&=
        \Delta\prs{\mcal{F}} - 2 \ch_1\prs{\mcal{F}} . kH - \prs{kH}^2 + 4 \deg\prs{\mcal{L}}
        \\&<
        \Delta\prs{\mcal{F}} - 2 \ch_1\prs{\mcal{F}} . kH - \prs{kH}^2 + 2 \ch_1\prs{\mcal{F}} . kH
        \\&=
        \Delta\prs{\mcal{F}} - \prs{kH}^2
        \\&=
        p - k^2 H^2 \text{.}
    \end{align*}
    Since $H$ is ample we have $H^2 > 0$. Since $k \geq \sqrt{\frac{p+4C_X}{H^2}}$, we get
    \begin{align}\label{equation:bg-cor-2} 
        \Delta\prs{\mcal{F}'}+4C_X < p - k^2 H^2 +4C_X\leq 0 \text{.}
    \end{align}
    Hence, $\mcal{F}'$ does not satisfy the BGK inequality and is therefore not $\mu_H$-semistable by \cref{Thm:MainThm}.

     By \eqref{corollary:bg-cor-1}, there's a unique rank $1$ reflexive subsheaf $\mcal{O}_X\prs{D}$ of $\mcal{F}'$ and an exact sequence
    \begin{align*}
        0 \to \mcal{O}_X\prs{D} \to \mcal{F}' \to \mcal{M} \to 0
    \end{align*}
    with $\ch_1\prs{\mcal{M}} = D'$, $\prs{D - D'}^2 > 4C_X$ and $\prs{D - D'}.H > 0$.
    Moreover, by \eqref{eq:2-step HNfiltration} it follows from the short exact sequence
    \[0 \to \mcal{M} \to \mcal{O}_X\prs{D'} \to \mcal{T} \to 0\]
    with $\mcal{T}$ zero-dimensional, that  $\ch_2\prs{\mcal{T}} \geq 0$ and that
    \begin{equation}\label{eqn:bg-cor-2-2}
       \Delta\prs{\mcal{F}'} = -\prs{D - D'}^2 + 4 \ch_2\prs{\mcal{T}} \text{.}
    \end{equation}

    As $\mcal{O}_X\prs{D}$ is also a subsheaf of $\mcal{F}$, so by stability $\mu_H\prs{\mcal{O}_X\prs{D}} < \mu_H\prs{\mcal{F}}$, i.e. $2 D . H < \ch_1\prs{\mcal{F}} . H$.
    Hence $m := -\prs{2D - \ch_1\prs{\mcal{F}}}.H > 0$.
    Moreover, the Hodge Index Theorem gives $\prs{2D - \ch_1\prs{\mcal{F}}}^2 H^2 \leq m^2$, or equivalently $\prs{2D - \ch_1\prs{\mcal{F}}}^2 \leq \frac{m^2}{H^2}$.

    Now, $D + D' = \ch_1\prs{\mcal{F}'} = \ch_1\prs{\mcal{F}} - kH$ so
    \begin{equation}\label{equation:bg-cor-D-D'}
    D - D' = 2D - \ch_1\prs{\mcal{F}} + kH \text{.}    
    \end{equation}
    Using $H . \prs{D - D'} > 0$ we get
    \begin{equation} \label{equation:m-ineq}
        0 < m = -\prs{2D - \ch_1\prs{\mcal{F}}} . H < kH^2 \text{.}
    \end{equation}
    Combining \eqref{equation:bg-cor-2}, \eqref{eqn:bg-cor-2-2}, and \eqref{equation:bg-cor-D-D'} we get
    \begin{align*}
        p-k^2 H^2 +4C_X &>\Delta\prs{\FF'}+4C_X= -\prs{D - D'}^2 +4\ch_2(\TT)+4C_X
        \\&\geq
        -\prs{\prs{2D - \ch_1\prs{\mcal{F}}} + kH}^2+4C_X
        \\&=
        -\prs{2D - \ch_1\prs{\mcal{F}}}^2 - 2k \prs{2D - \ch_1\prs{\mcal{F}}}.H - k^2 H^2+4C_X
        \\&\geq
        -\frac{m^2}{H^2} + 2k m - k^2 H^2+4C_X \text{.}
    \end{align*}
    Hence
    $2km < \frac{m^2}{H^2} + p$ so $2 < \frac{m}{kH^2} + \frac{p}{km}$. By \eqref{equation:m-ineq} we get $\frac{m}{kH^2} < 1$ and by $p \leq k$ we get $\frac{p}{km} \leq \frac{1}{m} \leq 1$. Hence
    \[2 < \frac{m}{kH^2} + \frac{p}{m} < 1 + 1 = 2 \text{,}\]
    a contradiction. Thus $\left. \mcal{F} \right|_C$ is stable.
\end{proof}

\begin{Cor} \label{corollary:bg-cor-3}
    Let $X$ be a projective surface with at most rational double points and let $N$ be a nef Weil divisor such that $N^2>4C_X$. Then $H^1\prs{X,\mcal{O}_X\prs{N+K_X}} = 0$.
\end{Cor}

\begin{proof}
    As RDP singularities are Gorenstein, $K_X$ is Cartier, so $\OO_X(N+K_X)\cong\OO_X(N)\otimes\OO_X(K_X)$ and by Serre duality,
    $$H^1\prs{X,\OO_X\prs{N+K_X}}=\Ext^1\prs{\OO_X,\OO_X\prs{N+K_X}}=\Ext^1\prs{\OO_X\prs{N},\OO_X}^\vee.$$
    Letting $\LL=\OO_X\prs{N}$, it suffices to show that every extension
    \begin{equation} \label{equation:cor3-ses}
        0 \to \mcal{O}_X \to \mcal{F} \to \mcal{L} \to 0
    \end{equation}
    splits.
    Since $\mcal{O}_X, \mcal{L}$ are reflexive, so is $\FF$.
    
    As 
    $$\ch\prs{\FF}=\prs{2,N,\frac{1}{2}N^2}$$
    By assumption we get that
    \[\Delta\prs{\mcal{F}} +4C_X= \ch_1\prs{\mcal{F}}^2 - 4 \ch_2\prs{\mcal{F}} +4C_X=4C_X -N^2< 0 \text{.}\]
    Hence $\mcal{F}$ is not $H$-semistable for any ample $H$ by \cref{Thm:MainThm}. Let
    \begin{equation} \label{equation:cor3-ses2}
        0 \to \mcal{O}\prs{D} \to \mcal{F} \to \mcal{M} \to 0
    \end{equation}
    be the associated short exact sequence as in \Cref{corollary:bg-cor-1}.
    The composition $\mcal{O}_X\prs{D} \hookrightarrow \mcal{F} \to \mcal{L}$ must then be non-zero by stability as $\mu_H(\FF)=\frac{N.H}{2}>0$.  Since $\mcal{O}_X\prs{D}$ and $\LL$ are torsion-free and of rank one, the kernel is trivial. Hence there is an exact sequence
    \[0 \to \mcal{O}_X\prs{D} \to \mcal{L} \to \mcal{C} \to 0\]
    with $\mcal{C}$ the cokernel. Thus we can write $N=D+E$, where $E := \ch_1\prs{\mcal{C}}$ is effective.
    Then as in \Cref{corollary:bg-cor-1}, we get a short exact sequence
    \begin{align*}
        0 \to \mcal{M} \to \mcal{O}_X\prs{E} \to \mcal{T} \to 0
    \end{align*}
    for a zero-dimensional sheaf $\mcal{T}$, so that $\ch_2\prs{\mcal{T}} \geq 0$.
    Thus
    $$\ch_2\prs{\mcal{M}} = \ch_2\prs{\mcal{O}_X\prs{E}} - \ch_2\prs{\mcal{T}}
        = \frac{1}{2}E^2 - \ch_2\prs{\mcal{T}} \text{.}$$
    On the other hand, by \eqref{equation:cor3-ses2} we have
    \begin{align*}
        \ch_2\prs{\mcal{M}} &= \ch_2\prs{\mcal{F}} - \ch_2\prs{\mcal{O}_X\prs{N-E}}
        \\&=
        \frac{1}{2} N^2 - \prs{\frac{1}{2} N^2 - N.E + \frac{1}{2}E^2}
        \\&=
        N.E - \frac{1}{2} E^2 \text{.}
    \end{align*}
    Comparing the two terms, we get
    \begin{equation}\label{equation:cor3:E^2}
        E^2 = N.E + \ch_2\prs{\mcal{T}} \geq N.E \geq 0 \text{,}
    \end{equation}
    since $N$ is nef and $E$ is effective.
    Now, \Cref{corollary:bg-cor-1} tells us that for every ample divisor $H$, $$\prs{2D - N}.H =\prs{D - E}.H > 0.$$
    It follows that
    \begin{align*}
        N^2 - 2 N.E
        &=
        2\prs{N - E}.N - N^2
        \\&=
        2D.N - N^2
        \\&=
        \prs{2D - N}.N \geq 0
    \end{align*}
    as $N$ is nef and thus a limit of ample divisors.  Thus    \begin{equation}\label{equation:cor3:L^2}
        N^2 \geq 2 N.E \text{,}
    \end{equation}
    and by \eqref{equation:cor3:E^2} we get
    \begin{equation}
        \prs{N^2}\prs{E^2} \geq 2 \prs{N.E} \prs{E^2} \geq 2 \prs{N.E}^2 \text{.}
    \end{equation}
    By the Hodge Index Theorem, $\prs{N^2} \prs{E^2} \leq \prs{N.E}^2$, hence we get $2\prs{N.E}^2 \leq \prs{N.E}^2$, which implies $N.E = 0$. By the Hodge Index Theorem again we get $E^2 \leq 0$ and thus $E^2 = 0$.  That is, $E$ is numerically equivalent to $0$. As $E$ is effective, we get $E = 0$, so in particular $\mcal{L} \cong \mcal{O}\prs{D} \hookrightarrow \mcal{F}$ and we get a nontrivial map $\mcal{L} \hookrightarrow \mcal{F} \to \mcal{L}$.  After adjusting by a scalar we can assume this map to be the identity, so the exact sequence \eqref{equation:cor3-ses} splits as desired.
\end{proof}

\bibliographystyle{plain}
\bibliography{NSF_Research_Proposal}

\begin{thebibliography}{10}

\bibitem{AB13}
D.~Arcara and A.~Bertram.
\newblock Bridgeland-stable moduli spaces for {$K$}-trivial surfaces.
\newblock {\em J. Eur. Math. Soc. (JEMS)}, 15(1):1--38, 2013.

\bibitem{BandoSiu:KahlerBG}
S.~Bando and Y.-T. Siu.
\newblock Stable sheaves and {E}instein-{H}ermitian metrics.
\newblock In {\em Geometry and analysis on complex manifolds}, pages 39--50.
  World Sci. Publ., River Edge, NJ, 1994.

\bibitem{BMT14a}
A.~Bayer, E.~Macr\`{i}, and Y.~Toda.
\newblock Bridgeland stability conditions on threefolds {I}:
  Bogomolov-{G}ieseker type inequalities.
\newblock {\em J. Algebraic Geom.}, 23(1):117--163, 2014.

\bibitem{Bri08}
T.~Bridgeland.
\newblock Stability conditions on {K3} surfaces.
\newblock {\em Duke Math. J.}, 141(2):241--291, 2008.

\bibitem{BH93}
W.~Bruns and J.~Herzog.
\newblock {\em Cohen-{M}acaulay rings}, volume~39 of {\em Cambridge Studies in
  Advanced Mathematics}.
\newblock Cambridge University Press, Cambridge, 1993.

\bibitem{Eis95}
D.~Eisenbud.
\newblock {\em Commutative algebra}, volume 150 of {\em Graduate Texts in
  Mathematics}.
\newblock Springer-Verlag, New York, 1995.
\newblock With a view toward algebraic geometry.

\bibitem{Esnault:Reflexive}
H.~Esnault.
\newblock Reflexive modules on quotient surface singularities.
\newblock {\em J. Reine Angew. Math.}, 362:63--71, 1985.

\bibitem{Fri98}
R.~Friedman.
\newblock {\em Algebraic surfaces and holomorphic vector bundles}.
\newblock Universitext. Springer-Verlag, New York, 1998.

\bibitem{Fulton}
W.~Fulton.
\newblock {\em Intersection theory}, volume~2 of {\em Ergebnisse der Mathematik
  und ihrer Grenzgebiete. 3. Folge. A Series of Modern Surveys in Mathematics
  [Results in Mathematics and Related Areas. 3rd Series. A Series of Modern
  Surveys in Mathematics]}.
\newblock Springer-Verlag, Berlin, second edition, 1998.

\bibitem{Har67}
R.~Hartshorne.
\newblock {\em Local cohomology}, volume 1961 of {\em A seminar given by A.
  Grothendieck, Harvard University, Fall}.
\newblock Springer-Verlag, Berlin-New York, 1967.

\bibitem{Har80}
R.~Hartshorne.
\newblock Stable reflexive sheaves.
\newblock {\em Math. Ann.}, 254(2):121--176, 1980.

\bibitem{Hartshorne:GeneralizedDivisors}
R.~Hartshorne.
\newblock Generalized divisors on {G}orenstein schemes.
\newblock In {\em Proceedings of {C}onference on {A}lgebraic {G}eometry and
  {R}ing {T}heory in honor of {M}ichael {A}rtin, {P}art {III} ({A}ntwerp,
  1992)}, volume~8, pages 287--339, 1994.

\bibitem{HL10}
D.~Huybrechts and M.~Lehn.
\newblock {\em The geometry of moduli spaces of sheaves}.
\newblock Cambridge Mathematical Library. Cambridge University Press,
  Cambridge, second edition, 2010.

\bibitem{Kol13}
J.~Koll{\'a}r.
\newblock {\em Singularities of the minimal model program}, volume 200 of {\em
  Cambridge Tracts in Mathematics}.
\newblock Cambridge University Press, Cambridge, 2013.
\newblock With a collaboration of S{\'a}ndor Kov{\'a}cs.

\bibitem{KM98}
J.~Koll{{\'a}}r and S.~Mori.
\newblock {\em Birational geometry of algebraic varieties}, volume 134 of {\em
  Cambridge Tracts in Mathematics}.
\newblock Cambridge University Press, Cambridge, 1998.
\newblock With the collaboration of C. H. Clemens and A. Corti, Translated from
  the 1998 Japanese original.

\bibitem{koseki:bg-pos}
N.~Koseki.
\newblock {On the Bogomolov–Gieseker Inequality in Positive Characteristic}.
\newblock {\em International Mathematics Research Notices}, page rnac260, 09
  2022.

\bibitem{Langer:SemistablePositiveChar}
A.~Langer.
\newblock Semistable sheaves in positive characteristic.
\newblock {\em Ann. of Math. (2)}, 159(1):251--276, 2004.

\bibitem{Langer:HiggsPositiveChar}
A.~Langer.
\newblock Bogomolov's inequality for {H}iggs sheaves in positive
  characteristic.
\newblock {\em Invent. Math.}, 199(3):889--920, 2015.

\bibitem{Mat02}
K.~Matsuki.
\newblock {\em Introduction to the {M}ori program}.
\newblock Universitext. Springer-Verlag, New York, 2002.

\bibitem{Mukai:CounterExamples}
S.~Mukai.
\newblock Counterexamples to {K}odaira's vanishing and {Y}au's inequality in
  positive characteristics.
\newblock {\em Kyoto J. Math.}, 53(2):515--532, 2013.

\bibitem{Raynaud:Counterexample}
M.~Raynaud.
\newblock Contre-exemple au ``vanishing theorem'' en caract\'{e}ristique
  {$p>0$}.
\newblock In {\em C. {P}. {R}amanujam---a tribute}, volume~8 of {\em Tata Inst.
  Fund. Res. Studies in Math.}, pages 273--278. Springer, Berlin-New York,
  1978.

\bibitem{Voisin:Counterexamples}
C.~Voisin.
\newblock A counterexample to the {H}odge conjecture extended to {K}\"{a}hler
  varieties.
\newblock {\em Int. Math. Res. Not.}, (20):1057--1075, 2002.

\bibitem{Wu:BGSingular}
X.~Wu.
\newblock The {B}ogomolov inequality on a singular complex space.
\newblock 2023.

\end{thebibliography}
\end{document}